\documentclass[a4paper]{amsart}
\usepackage{amsmath,caption,booktabs,lipsum}
\numberwithin{equation}{section}
\usepackage{amsthm}
\usepackage{amsfonts}
\usepackage{amssymb}
\usepackage{mathrsfs}
\usepackage{tikz}
\usepackage{environ}
\usepackage{elocalloc}
\usepackage{url}
\usepackage{caption}
\usepackage{graphicx}
\usepackage{CJKutf8}
\usepackage[normalem]{ulem}
\usepackage{xcolor}
\usepackage{etoolbox}
\usepackage{caption}
\usepackage{mathtools}
\usepackage{dcpic}
\usepackage{tikz-cd}
\usepackage[bottom]{footmisc}
\usepackage{hyperref}
\usepackage{bookmark}
\usepackage{enumitem}
\usepackage{stmaryrd}
\usetikzlibrary{positioning}
\usetikzlibrary{arrows,chains,positioning,scopes,quotes}
\usetikzlibrary{decorations.markings}

\newcommand{\ai}{\alpha}
\newcommand{\be}{\beta}
\newcommand{\Ga}{\Gamma}
\newcommand{\ga}{\gamma}	

\newcommand{\de}{\delta}

\newcommand{\e}{\epsilon}
\newcommand{\hm}{\mathcal{H}}
\newcommand{\lam}{\lambda}

\newcommand{\om}{\omega}
\newcommand{\Om}{\Omega}
\newcommand{\fr}{\textnormal{FocalRad}}
\newcommand{\si}{\sigma}
\newcommand{\Si}{\Sigma}
\newcommand{\vp}{\varphi}
\newcommand{\rh}{\rho}
\newcommand{\ta}{\theta}
\newcommand{\Ta}{\Theta}

\newcommand{\cms}{\operatorname{comass}}

\newcommand\res{\mathop{\hbox{\vrule height 7pt width .3pt depth 0pt
			\vrule height .3pt width 5pt depth 0pt}}\nolimits}

\newcommand{\cd}{\cdots}
\newcommand{\T}{\mathbf{T}}
\newcommand{\s}{\subset}
\newcommand{\es}{\varnothing}
\newcommand{\cp}{^\complement}

\newcommand{\TO}{\mathbb{T}}

\newcommand{\la}{\langle}
\newcommand{\ra}{\rangle}
\newcommand{\ov}[1]{\overline{#1}}
\newcommand{\no}[1]{\left\lVert#1\right\rVert}
\newcommand{\cu}[1]{\left\llbracket#1\right\rrbracket}

\DeclarePairedDelimiter{\ri}{\la}{\ra}
\newcommand{\dis}{\textnormal{dist}}

\newcommand{\du}{^\ast}
\newcommand{\pf}{_\ast}
\newcommand{\spt}{\textnormal{spt}}

\newcommand{\m}{^{-1}}
\newcommand{\ts}{\otimes}

\newcommand{\cm}{\mathbf{M}}

\newcommand{\w}{\wedge}

\newcommand{\pd}{\partial}

\newcommand{\na}{\nabla}

\newcommand{\R}{\mathbb{R}}
\newcommand{\Z}{\mathbb{Z}}

\newcommand{\C}{\mathbb{C}}

\newcommand{\id}{\textnormal{id}}

\makeatletter
\def\thm@space@setup{%
	\thm@preskip=0.2cm plus 0cm minus 0cm
	\thm@postskip=\thm@preskip 
}
\makeatother
\theoremstyle{plain}
\newtheorem{thm}{Theorem}[section]
\newtheorem{lem}[thm]{Lemma}

\newtheorem{cor}{Corollary}

\newtheorem*{thmm}{Theorem}

\theoremstyle{definition}
\newtheorem{defn}{Definition}[section]

\theoremstyle{remark}
\newtheorem{rem}{Remark}

\title[Fractal singular sets II]{Area-minimizing surfaces with fractal singular sets: prevalence, moduli space, and refinements on strata}
\author{Zhenhua Liu}
\dedicatory{Dedicated to Xunjing Wei}
\begin{document}
\vspace{-3em}
	\begin{abstract}
		This paper is a continuation of our work in \cite{ZL}, about a conjecture of Almgren on area-minimizing surfaces with fractal singular sets. First, we prove that area-minimizing surfaces with fractal singular sets are prevalent on the homology level on general manifolds. They exist even in metrics arbitrarily close to the flat metric on a torus. Second, we partially determine the moduli space of area-minimizing currents near the ones with fractal singular set that we constructed in this and previous work. Around strata of codimension at least three, the fractal self-intersections deform into transverse immersions under generic perturbations of metric. The result is sharp and we construct concrete example to show that our fractal singularities do not necessarily completely dissolve under generic perturbations. Finally, we improve the results in \cite{ZL} to give examples with different strata limiting to each other and refine the first stratum to fractal subsets of finite combinatorial graphs.
	\end{abstract}	\maketitle
	\tableofcontents
	\section{Introduction}
	In this paper, we continue our work in \cite{ZL} on solving the following conjecture of Almgren (Problem 5.4 in \cite{GMT}),
	
	"\textit{Is it possible for the singular set of an area minimizing integer (real) rectifiable current to be a Cantor type set with possibly non-integer Hausdorff dimension?}"  
	
	The reader can refer to the introduction of \cite{ZL} for a detailed discussion of the history of the problem. The main theorem obtained in \cite{ZL} is as follows,
	\begin{thmm} (Theorem 1.2 in \cite{ZL})
		For any integers $n\ge 2,$ $0\le j\le n-2,$ any sequence of smooth compact (not necessarily connected)  Riemannian manifolds $N_j$ of dimension $j,$ let $K_j$ be compact subsets of $N_j,$ and $k$ be the smallest number with $K_k$ nonempty. There exists a smooth compact Riemannian manifold $M^{2n-k+1}$, and a smooth calibration $n$-form $\phi$ on it, so that $\phi$ calibrates an irreducible homologically area minimizing surface $\Si^n$ in $M.$ The singular set of $\Si$ is $K_0\cup\cd\cup K_{n-2},$ and the $j$-symmetric part of the $j$-th strata of the Almgren stratification of the singular set of $\Si$ is precisely $K_j.$ 
	\end{thmm}
	However, this still leaves open several very crucial questions. First, the manifolds we constructed have special topology and it is not known whether such phenomena can happen in general manifolds. In this direction, we prove the following theorem.
	\begin{thm}\label{prevalence}
		On any compact closed not necessarily orientable smooth manifold, there exist homologically area-minimizing surfaces with fractal singular set in some metrics, as long as the rational homology of $M$ is nontrivial in some homology class of dimension and codimension larger than $2$.
	\end{thm}
	\begin{rem}
		The dimension requirement is sharp, since $2$-dimensional area-minimizing surfaces are branched minimal immersions by \cite{DSS3}.
	\end{rem}
	This follows from the following more general result about prescribing singular sets.
	\begin{thm}\label{prev}
		For any $m+n$-dimensional smooth compact closed not necessarily orientable manifold $M,$ let $[\Si]$ be an $n$-dimensional integral homology class with $m,n\ge 3$. Let $K$ denote arbitrary closed subsets of the standard spheres $S^{n-2}$ of dimension $n-2.$ If $[\Si]$ can be represented by an embedded submanifold, then in a smooth metric $g$, the $n-2$-th strata of the Almgren stratification of the unique area-minimizing integral current $T$ in the class $[\Si]$ is precisely $K$.
	\end{thm}
\begin{rem}
	By \cite{RT}, this happens unconditionally for $n\le 6.$
\end{rem}
	Of course, with all these abstract examples, one is natural to ask if such phenomena always happen in exotic metrics not encountered in real life. The answer is a striking no. Such phenomena can happen arbitrarily close even to the standard metric in the standard flat torus.
	\begin{thm}\label{tor}
		Let $\TO^{m+n}=\R^{m+n}/\Z^{m+n}$ be the standard $m+n$ flat torus with the standard flat metric $\de$ and $m\ge 3,n\ge 3$. Suppose $\TO_1,\TO_2$ are two $m$-dimensional subtorus intersecting non-transversely and orthogonally along a subtorus $\TO_3$ of dimension $1\le l\le m-2.$ Let $K$ be an arbitrary closed subset of $\TO_3,$ then there exists a sequence of smooth metric $g_j\to \de$ converging smoothly to $\de,$ so that in  each $g_j$ a homological area-minimizer in the homology class $[\TO_1]+[\TO_2]$ has $K$ as its singular set.
	\end{thm}
\begin{rem}
At first sight, it might seem impossible that singular set $K$ can jump to become $\TO_3$ in the limit $g_j\to \de.$ In each $g_j,$ $K$ is defined  by a smooth function $j\m f$ defined on $\TO_3$ with $f\m(0)=K$. Thus limiting to $\de,$ we naturally obtain $\TO_3$ as the singular set
\end{rem}
	Second, it is not known a prior if the fractal singular sets we obtained will disappear under generic perturbations of metrics. We obtain an almost complete picture as follows.
	\begin{thm}\label{mod}
		For all the homologically area-minimizing currents (integral or mod $v$ with $v\ge 2$) we obtained in \cite{ZL} and Theorem \ref{prev}, the fractal singular set in strata of codimension larger than  $2$ mollify into submanifolds formed by transverse intersections under generic perturbations of metrics. 
	\end{thm}
	\begin{rem}
		The codimension $2$ strata are much trickier. Conjecturally, under generic perturbations, the singular set will be submanifolds, but not necessarily formed by transverse intersection. (The heuristic is that $zw=0$ smoothify into $zw=\e$ in $\C^2$.)
	\end{rem}
	\begin{rem}
		Technically speaking, our constructions are always parameterized by smooth functions $f$ prescribing the intersection set. The result requires $f$ to have a small $C^2$ norm, which can always be achieved.
	\end{rem}
	The above theorem is a direct corollary of the following Allard-type theorem, which might be of independent interest.
	\begin{thm}\label{modl}
Let $d,c\ge 3,v\ge 2$ be integers and $T$ be a $d$-dimensional integral or mod $v$ homologically area-minimizing current in a $d+c$-dimensional compact not necessarily orientable Riemannian manifold $M^{d+c}$ with a metric $g$. Suppose $T$ satisfies the following conditions.
		\begin{itemize}
			\item $T$ comes from the immersion of a compact not necessarily connected not necessarily orientable manifold $\Si$.
			\item The immersion intersect itself cleanly  (Definition \ref{cleanint}) at compact closed not necessarily orientable submanifolds $L_1,\cd,L_k.$ 
			\item Near the self-intersection set of $T$, $T$ decomposes into the sum of two not necessarily connected embedded submanifolds.
			\item $d-\dim L_j\ge 3$ for all $j,$
			\item $\pd T$ does not intersect the self intersection set, and $T$ does not intersect $\pd M,$
			\item The angles of self-intersections (Definition \ref{intangm}) is always larger than $$2\arctan\frac{2}{\sqrt{(d-\dim L_j)^2-4}}.$$
		\end{itemize} Then there exists an open neighborhood $\Om_T$ of $g$ in the space of Riemannian metrics and an open set $U\s M$  containing the support of $T,$ with the following properties. 
		\begin{itemize}
			\item For any sequence $g_j\to g$ with $\{g_j\}\s \Om_T,$ let $T_j$ be any sequence of area-minimizing integral current in the metric $g_j,$ supported in $U$ and homologous to $T$. Then $T_j$ converges to $T$ smoothly as immersions of $\Si.$
			\item If $T$ is the unique minimizer, then we can take $U=M.$
		\end{itemize}
	\end{thm}
	\begin{rem}
		\textbf{Warning}: the self-intersection set of nearby minimizers in the moduli space can be fractal.
	\end{rem}
	\begin{rem}
		In the integral and mod $2$ case, the dimension $d$ lower bound is sharp due to \cite{BW} and \cite{BW2}. In the integral case, the codimension $d-\dim L_j$ lower bound is sharp by considering holomorphic subvarieties. For instance, projective variety in $\C\mathbb{P}^4$ induced by $\{z_1=0\}\cap \{z_2z_3=0\}\s \C^5$  with can be smoothed into $\{z_1=\e z_4\}\cap\{z_2z_3=\e z_4^2\}.$
	\end{rem}
	\begin{rem}
		The angle condition is of sharp asymptotic order in view of Lawlor's necks (\cite{GLi}), where the topology of nearby area-minimizing currents can change and self-intersection singularities can disappear when the angle equals $\frac{\pi}{d}.$ 
	\end{rem}
	Given Theorem \ref{mod}, the natural question to ask is whether the singularities will disappear totally under generic perturbations. Unfortunately, the answer is no.
	\begin{cor}\label{cte}
		For all the homologically area-minimizing integral currents $T$ in metric $g_T$ we obtained in \cite{ZL} and Theorem \ref{prev}, there exists an open sets $U$ in the ambient manifold and open sets $\Om_T$ in the space of smooth Riemannian metrics , so that
		\begin{itemize}
			\item 
			$\spt T\s U,g_T\in\ov{\Om_T},$
			\item for $g\in \Om_T$, the singular set of area-minimizing currents $T_g$ in $U$ is always nonempty and consists of no-empty transverse intersections for generic $g$.
			\item in case $T$ is the unique minimizer in its homology class, we can take $U=M.$
		\end{itemize}
	\end{cor}
	Thus, the singularities will not disappear totally under generic perturbations, unless transversality gives an empty intersection. However, as remarked in \cite{ZL}, under specific perturbations of the metric, the singularity will disappear totally.
	
	Another caveat of Theorem 1.2 in \cite{ZL} is that the singular sets of different strata are disjoint. As an improvement, we show that different strata of fractal dimension can limit to each other.
	\begin{thm}\label{torus}
		\begin{itemize}\hspace{2em}
			\item Let $\TO_1,\cd,\TO_k\s \TO^{d+c}$ be $k$ pairwise orthogonal $d$-dimensional subtorus of the standard $d+c$-dimensional torus, with $k\ge 2,d\ge 3,c\ge 2,$
			\item $\TO_1,\cd,\TO_k\s \TO^{d+c}$ are in general position, i.e., having linearly orthogonal complements (Definition 19 in \cite{BW2}),
			\item Let $N^n$ be a compact closed connected orientable smooth submanifold of a compact $m+n$-dimensional closed not necessarily orientable smooth Riemannian manifold $M,$
			\item $N$ belongs to a nontrivial torsion-free homology class in $M.$
			\item $N_1,\cd,N_k$ are $C^2$ small perturbations of $N.$
		\end{itemize}
		Then there exists a smooth metric on $\TO^{d+c}\times M,$ so that 
		\begin{align*}
			T=	\sum_j \TO_j\times N_j
		\end{align*}
		is homologically area-minimizing. 
	\end{thm}
	\begin{rem}
		The singular set of  $T$ is the union of pairwise intersection among $\{\TO_j\times N_j\}.$	The intersection of $\{\TO_j\}$ alone can already be very complicated, with different strata limiting to other ones.  The addition of $N_j$ factors makes the singular set more complicated. For example, we can let different $N_j$ having fractal intersection, e.g., graph over each other with fractal zero sets. 
	\end{rem}
	Finally, since one-dimensional compact manifolds are just a disjoint union of circles, the result for the 1st strata in the Almgren stratification does not give as much interesting ambient topology for the singular set as other strata. We improve the result in the first strata to the following.
	\begin{thm}\label{slag}
		Given any (not necessarily connected) finite graph $G$ in the combinatorial sense, let $K$ be a compact subset of $G,$ which contains a neighborhood of each vertex. Let $M$ be a smooth $6$-dimensional manifold with a nontrivial $3$-dimensional homology class $[\Si].$ Then there exists a smooth metric $g$ on $M,$ so that the unique area-minimizing current $T$ in the class $[\Si]$ has $K$ as the singular set. Moreover  $K$ is precisely the $1$-strata in the Almgren stratification of $T$. 
	\end{thm}
	\begin{rem}
		Recall that by the results in \cite{NV}, the $j$-th strata of a stationary varifold is $j$-rectifiable. Since $1$-rectifiable sets are covered a.e. by countable unions of parts of Lipschitz curves, we can say roughly that $1$-rectifiable sets are subsets of infinite degree combinatorial graphs. Thus, the above theorem is genuinely much stronger than Theorem 1.2 in \cite{ZL}, and gets much closer to realizing all $1$-rectifiable sets. It is tempting to conjecture that for any Riemannian manifold $M,$ all $1$-rectifiable sets can be realized as the first strata in the Almgren stratification of an area-minimizing current.
	\end{rem}
	\subsection{Sketch of the proof}
	All proofs are technical extensions of the results in \cite{ZL} and \cite{ZL1} by considering more general geometric models. 
	\section*{Acknowledgements}
	The author acknowledges the support
	of the NSF through the grant FRG-1854147.
	I cannot thank my advisor Professor Camillo De Lellis enough for his unwavering support while I have been recovering from illness. I feel so lucky that I have Camillo as my advisor. I would also like to thank Professor Frank Morgan for his encouragement and support.
	\section{Basic notions and preliminaries}
	In this section, we will fix our notations and give some useful lemmas.
	\subsection{Inner product vector spaces}
	Let $V$ be a $m$-dimensional vector space equipped with an inner product. Suppose $P$ is an oriented $p$-dimensional linear subspace of $V.$ Then we will also use $P$ to denote the $p$-dimensional simple vector
	\begin{align*}
		P=e_1\w\cd\w e_p,
	\end{align*}where $e_1,\cd,e_p$ is a set of oriented orthonormal frame of $P.$ \begin{defn}
		The dual form of $P$ is defined as
		\begin{align}
			P\du=e_1\du\w\cd\w e_p\du,
		\end{align}
		where $e_j\du$ is the $1$-form $\ri{e_j,\cdot}$.
	\end{defn} 
	It is straightforward to verify that $P$ and $P\du$ do not rely on the orthonormal frame we choose.
	\subsection{Euclidean space and torus}
	Let $\R^{m+n}$ be a Euclidean space endowed with the standard inner product. Most of the time, it is convenient to split $\R^{m+n}$ into a product of subspaces, i.e., $\R^m\times \R^n.$ We will use $x,y,a,b,$ etc, to denote coordinate labels of each factor.
	\begin{defn}
		The standard torus $\TO^{m+n}$ formed from $\R^{m+n}$ is defined to be the Riemannian quotient of $$\TO^{m+n}=\R^{m+n}/\Z^{m+n},$$ where the standard $\Z^{m+n}$ acts by translation. A $d$-dimensional subtorus $\TO$ of $\TO^{m+n}$ is an embedded $d$-dimensional torus that can be lifted to a $d$-dimensional subspace in $\R^{m+n}.$  
		By abusing notations, we will sometimes refer to the subtorus by its lifted subspace and vice versa. For subtorus passing through zero, we can talk about the span of several subtorus by considering the span of their lifts and the corresponding subtorus, etc.
	\end{defn}
	It is straightforward to verify that 
	\begin{align*}
		\TO^{m+n}=\TO^m\times\TO^n,
	\end{align*}
	as Riemannian products.
	
	A simple fact we will use is that any constant differential form on $\R^{m+n}$ descends onto $\TO^{m+n}.$ 
	
	\subsection{Manifolds and neighborhoods}
	We will reserve $M$ to denote an ambient smooth Riemannian manifold. If not mentioned, then we do not put any orientability or compactness assumptions on $M.$ Submanifolds will be denoted by $X,Y,L,N$, etc. We will use the following sets of definitions and notations often.
	\begin{defn}
		Let $N$ be a smooth submanifold of a Riemannian manifold $M.$
		\begin{itemize}
			\item $B_r^M(N)$ denotes a tubular neighborhood of $N$ inside $M$ in the intrinsic metric on $M$ of radius $r$.
			\item $\T_p M$ denotes the tangent space to $M$ at $p.$ We will often regard $\T_p N$ as a subspace of $\T_p M.$
			\item $\exp^\perp_{N\s M}$ denotes the normal bundle exponential map of the inclusion $N\s M.$
			\item $\fr_N^M$ is the focal radius of $N$ in $M,$ i.e., the radius below which the normal bundle exponential map remains injective.
			\item $\pi_N^M$ denotes the nearest distance/normal bundle exponential map projection from $M$ to $N$ in the metric intrinsic of $M$ in $B_r^N(M)$ with $r$ less than the focal radius of $N$ inside $M$.
			\item $\dis_M(p,q)$ is the distance between two points with respect to the intrinsic measure on $M.$
			\item $d_N(p,q)$ is the distance between $p,q$ with respect to the intrinsic distance on $N.$
		\end{itemize}   
	\end{defn}
	\begin{defn}\label{angn}
		For the inclusion of submanifolds $L\s X\s M$ of different dimensions, we define the angle $\ta$ neighborhood of radius $r$  of $X$ inside $M$ with vertex $L$ as follows,
		\begin{align*}
			W_{(r,\ta)}^M(N,L)=\{p|d_M(p,L)\le r,d_M(p,X)/d_X(\pi^M_X(p),L)\le\tan\ta\}.
		\end{align*}
	\end{defn}
	\subsection{Immersions and intersections}
	\begin{defn}\label{cleanint}
		We say that two submanifolds $X,Y$ of dimension $d$ intersect cleanly at a submanifold $L$ of dimension $l$ if for every point of $L,$ there exists a coordinate neighborhood so that $X,Y$ become two coordinate $d$-dimensional planes intersecting at an $l$-dimensional plane. 
	\end{defn}
	If we have two planes $P_1$ and $P_2$ intersecting non-trivially inside a Euclidean space, then $P_1\cup P_2$ splits off as a product of their intersection $P_1\cap P_2$ with two planes $P_1'$ and $P_2'$ intersecting only at the origin in a lower dimensional Euclidean space.
	\begin{defn}\label{ints}
		The intersection angle $\Ta(P_1,P_2)$ is defined to be \begin{align*}
			\Ta(P_1,P_2)=\inf_{v\in P_1',w\in P_1'}\arccos\frac{|\ri{v,w}|}{|v||w|}.
		\end{align*}
	\end{defn}
	\begin{defn}\label{intangm}\hspace{2em}
		Let $X$ and $Y$ be two submanifolds of $M$ intersecting cleanly along a submanifold $L,$ then the intersection angle of $X$ and $Y$ is defined as follows,
		\begin{align*}
			\Ta(X,Y)=\inf_{p\in L}\Ta(\T_pX,\T_pY).
		\end{align*}
	\end{defn}
	The connection between intersection angle and angle neighborhoods is as follows. 
	\begin{defn}
		Let $L$ be a compact boundaryless submanifold of a compact submanifold $X.$ Define exponentiated angle neighborhood of $X$ at $L$ as follows,
		\begin{align*}
			\exp^{\perp}W_{(r,\ta)}^M(X,L)=\bigcup_{p\in L}\exp^\perp_{X\s M}W_{(r,\ta)}^{\T_pM}(\T_p X,\T_p L).
		\end{align*} 	 
	\end{defn}
	\begin{rem}
		It is clear from the definitions that 
		\begin{align*}
			W_{(r,\ta)}^{\T_pM}(\T_p X,\T_p L)=\bigcup_{\substack{V\textnormal{ is a }\dim X\textnormal{ subspace of }\T_p M\\ \Ta(V,\T_p X)\le \ta}}V.
		\end{align*}
	\end{rem}
	\begin{lem}\label{intang}
		For any $\frac{\pi}{2}>\ta_0>0,$ there exists $r_0>0,$ so that we have 
		\begin{align*}
			W_{(r,\ta)}^M(X,L)\s \exp^\perp W_{(r,\arctan[\tan\ta+O({r_0})])}^M(X,L),
		\end{align*}
		with big $O$ depending geometrically on $L,X,M.$
	\end{lem}
	\begin{proof}
		By definition, it suffices to prove this for $r=r_0,\ta=\ta_0.$ First, let $r_0<\min\{\fr_L^X,\fr_L^M,\fr_X^M\}.$ For $r_0$ small enough, $X$ is graphical over $\exp^\perp_{L\s M} \T X,$ with zero differential at $L.$ For any point in $L,$ take a normal coordinate with respect to the intrinsic metric on $L,$ then establish a Fermi coordinate using orthonormal frames in the normal bundle of $L.$ By compactness of $L,$ take a finite subcover. Then direct calculation using Taylor's theorem gives that $\dis_M(p,q),\dis_X(p,q)=(1+O(\max{|p|,|q|}))|p-q|$ in the coordinate patches, with $|\cdot|^2$ the sum of coordinate components squared. Thus, for a point $p\in \exp^\perp_{L\s M} \T X$ of distance $\rh$ to $L$, there is a point $p_X\in X$ with $\dis_X(L,p_X)=O(\rh)^2$ and vice versa. Suppose $q$ is a point in $W_{(r,\ta)}.$ Define
		$
		q_X=\pi_X^M(q),q_L=\pi_L^X(q_X),v=[\exp^\perp(L\s M)]\m(q).
		$ Let $\ell$ be the point in $L$ with $v\in \T_\ell M.$ Then define $v_X=\pi_{\T_\ell X}^{\T_\ell M}(v),
		p_X=\exp^\perp_{L\s M}(v_X).$ By the previous discussion and the fact that $\pi$ is nearest distance projection, we deduce that $\dis_M(p_X,q_X)=O(\dis_M(q_X,L)^2)=O({\ri{v_X,v_X}})$. We have
		\begin{align*}
			&\dis_X(q_X,q_L)\ge {\dis_M(q_X,q_L)}\ge{\dis_M(q_L,p_X)-\dis_M(p_X,q_X)}\ge\\& \dis_M(L,p_X)+O(\dis_M(q_X,q_L))=\dis_M(\ell,p_X)+O(\dis_M(q_X,q_L)).
		\end{align*}
		Since $\dis_M(q_X,q_L)=O(r_0),$ rearranging gives $$\dis_X(q_X,q_L)\ge(1+O(r_0))\dis_M(\ell,p_X)=(1+O(r_0))\dis_{T_\ell M}(p_X,0).$$ On the other hand,
		\begin{align*}
			\dis_M(q,q_X)\le \dis_M(q,p_X)=(1+O(r_0))\dis_{\T_\ell M}(v,v_X).
		\end{align*}
		This yields
		\begin{align*}
			\frac{\dis_M(q,q_X)}{\dis_X(q_X,q_L)}\le(1+O(r_0))\frac{\dis_{\T_\ell M}(v,v_X)}{\dis_{\T_\ell M}(p_X,0)}.
		\end{align*}
		This yields the desired result.
	\end{proof}
	\begin{lem}
		Let $X,Y$ be compact closed (not necessarily orientable) submanifolds intersecting at a compact boundaryless submanifold $L,$ with intersection angle $\Ta(X,Y)>0.$ Then for any $0<\ta_0<\frac{1}{2}\Ta(X,Y),$ there exists $r_0$, which depends continuously on the geometry of $X,Y,M$ and $\frac{1}{2}\Ta(X,Y)-\ta_0,$ so that  we have
		\begin{align*}
			W_{(r,\ta)}^M(X,L)\cap W_{(r,\ta)}^M(Y,L)=L,
		\end{align*}
		for any $0\le r\le r_0,0<\ta\le\ta_0.$
	\end{lem}
	\begin{proof}
		This follows directly from Lemma \ref{intang}.
	\end{proof}
	\subsection{Angle neighborhoods lemma}
	\begin{lem}\label{ang}
		The basic set up is as follows.
		\begin{itemize}
			\item 
			$L^l\s X^d,Y^d\s N^n\s M^{m+n}$ is a sequence of inclusion of compact, not necessarily orientable, and not necessarily boundaryless smooth Riemannian manifolds. 
			\item $X$ and $Y$ intersects transversely inside $N$ along $L$ with $\pd L=\es,L\cap\pd X=\es,L\cap\pd Y\not=\es,$
			\item the intersection angle of $X$ and $Y$ satisfies $\frac{\pi}{2}>\ta_0>0.$
		\end{itemize}
		Then there exists $\e,r_0>0,$ so that whenever $X',Y'$ are smooth perturbations of the embeddings $X$ and $Y$ inside $M$ with $C^2$, which are $\e$ close in $C^2$ norm to the original embedding, we can find smooth $l$-dimensional submanifolds $L_X'\s X'$ and $L_Y'\s Y'$ both of which are perturbations of the embedding of $L$, so that they are $O(\e)$ close in $C^2$ norm to $L.$ Moreover, we have
		\begin{align*}
			W_{(r,\ta)}^M(X',L_X')\cap W_{(r,\ta)}^M(Y',L_Y')=L_X'\cap L_Y',
		\end{align*}
		for all $r\le r_1$ and $0<\ta<\ta_0/2.$
	\end{lem}
	\subsubsection{Setting up $L_X',L_Y'$}By our assumption, for $\e$ small enough, $\pi_N^M|_{X'}$ is an embedding because constant rank is an open condition, i.e., by using determinants of minors. Denote the projected image $\pi_N^M(X')$ by $$X'_N.$$ Similar statements hold for $Y'$ and denote $\pi_N^M(Y')$ by $Y'_N.$ 
	
	By dimension counting of $X,Y,N$, and Theorem 3.11 in \cite{GG} , for $\e$ small, we deduce that, $X'_N\cap Y'_N$ intersects transversely at an $l$-dimensional submanifold $L'$ Let $$L_X'=(\pi_N^M|_{X'})\m (L'),L_Y'=(\pi_N^M|_{Y'})\m (L').$$ 
	
	Our strategy is to prove first
	\begin{align}\label{goal}
		\pi_N^M(W_{(\frac{1}{2}r,\ta)}^M(X',L'_X))\s W_{(r,\arctan[(1+O(\sqrt{r_1}))\tan\ta])}^N(X'_N,L')
	\end{align}
	for some $\e<r_1<0, r\le r_1$, 	and do so for $Y'.$ In other words, the projection of angled neighborhoods of $X'$ around $L'$ is roughly contained in the angled neighborhood of the projection. The angle neighborhoods of the projection of $X'$ and $Y'$ only intersect at $L'.$ Thus, to prove the lemma, it suffices to show that $(\pi_N^M)\m (L')$ only intersects the angled neighborhood of $(X',L'_X)$ at $L_X'$
	\subsubsection{Setting up coordinate patches}
	For any point $a\in L'_X,$ first set up a normal coordinate system with coordinate label $\ell$ near $a$ on $L'_X$ in the intrinsic metric on $L'_X.$  Then adopt a Fermi coordinate with respect to the normals to $L'_X$ on $X_N'$ obtained by parallel transporting a fixed frame of the normal bundle at $a.$ Use coordinate label $x.$ Then repeat for the inclusion $X_N'\s N$ with coordinate label $\nu$. Repeat again for the inclusion $N\s M$ with coordinate label $\mu.$ Thus, any point in this layered coordinate system will be denoted by $(\ell,x,\nu,\mu).$
	
	It is straightforward to verify that  $\na_{v_j}v_l=0$ at the point $(0,0,0,0)$ for any vectors $v_j,v_l$ being coordinate vector fields. Taylor's theorem with remainder gives 
	\begin{align}\label{covder}
		\na_{v_j}v_l=O(r),
	\end{align} with $r=|\ell|^2+|x|^2+|\nu|^2+|\mu^2|$ and big $O$ only depending on the geometry of $N,L',X'.$
	
	\textbf{Warning!} Here $|\cdot|^2$ denoting sum of coordinate components squared. It does not denote the length measured by the metric. Though by Taylor's theorem with remainder, we always have 
	\begin{align}\label{eucnorm}\frac{|v+w|^2-|v-w|^2}{4}=g_{(\ell,x,\nu,\mu)}(v,w)+O(r)|v||w|,\\\label{lth}\dis_M(p,q),\dis_N(p,q),\dis_{X_N'}(p,q)=(1+O(\max\{|p|,|q|\}))|p-q|.
	\end{align} 	
	Furthermore, by Chapter 2 in \cite{AG}, $|\mu|$ equals the distance to $N,$ $|\nu|$ equals the distance to $X$ in $N,$ etc. At every step of taking the normal/Fermi coordinates we can ensure that new newly added label can reach the length of at least some $r_1>0$ in every direction by compactness and considering focus radius/injectivity radius. In other words all of 
	\begin{align}\label{r1}
		\{|\ell|\le r_1\}\times\{|x|\le r_1\}\times\{|\nu|\le r_1\}\times\{|\mu|\le r_1\}
	\end{align} is well-defined in this neighborhood. From now on we
	
	Since $L$ is compact, we can choose a finite subcover of these recursively defined Fermi coordinates. Call the collection $\{U_j\}.$
	
	By construction, in each coordinate $U_j,$ we have $\pi_N^M(\ell,x,\nu,\mu)=(\ell,x,\nu),$  and
	\begin{align}\label{angnei}
		W_{(r,\ta)}^N(X'_N,L')=\{|\ell|^2+|x|^2+|\nu|^2\le r^2,|\nu|/|x|\le \tan\ta\}.
	\end{align}
	Direct calculation using Taylor's theorem with remainder gives that
	\begin{align}\label{metdiff}
		g_{(\ell,x,\nu,\mu)}(v,w)= g_{(\ell,x,\nu,0)}(v,w)+O(|\mu|)|v||w|,
	\end{align} with $v,w$ extended constantly coordinate-wise and big $O$ depending on the geometry of $N,X',L'.$ 
	
	Finally, by our assumptions, $X'$ and $L'_X$ can be parameterized by
	\begin{align*}
		L'_X=\{(\ell,0,0,\mu_X(\ell,0))\}	,	X'=\{(\ell,x,0,\mu_X(\ell,x))\}.
	\end{align*}
	for smooth functions $\mu_X$ with $C^2$ norm of $O(\e).$ This implies that for any tangent vector $v$ to $X',$ we have
	\begin{align}\label{xmet}
		v_{\nu}=0,|v_{\mu}|=O(\e)|v_{(\ell,x)}|,
		g_{{X'}}(v,w)=g_{X'_N}(v,w)+O(\e)|v||w|.
	\end{align}
	Subscripts $\mu,(\ell,x)$, etc., means taking only the components in these coordinate directions.
	\subsubsection{Length of curves on $X'$ under projection}
	Let $\ga$ be an arbitrary smooth curve on $X'.$ Parameterize $\ga$ with $|\dot{\ga}|=1$. By (\ref{xmet}) and (\ref{eucnorm}), we have
	\begin{align*}
		\hm^1(\ga)=&\int\sqrt{g_{\ga(t)}(\dot{\ga},\dot{\ga})}dt\\
		=&\int\sqrt{g_{\pi_N^M\circ\ga(t)}(\dot{\ga},\dot{\ga})+O(\e)|\dot{\ga}|^2}dt\\
		=&\int\sqrt{g_{\pi_N^M\circ\ga(t)}(\dot{\ga}_{(\ell,x)},\dot{\ga}_{(\ell,x)})+g_{\pi_N^M\circ\ga(t)}(\dot{\ga}_{\mu},\dot{\ga}_{\mu})+O(r)+O(\e)}dt\\
		=&\int\sqrt{(1+O(\e))g_{\pi_N^M\circ\ga(t)}(\dot{\ga}_{(\ell,x)},\dot{\ga}_{(\ell,x)})+O(r_1)}dt.
	\end{align*}
	Here $r$ is defined just after (\ref{covder}) and $r_1$ is defined in (\ref{r1}). On the other hand, by (\ref{xmet}) $|\dot{\ga}_{(l,x)}|=(1+O(\e))|\dot{\ga}|=1+O(\e).$ This gives $g_{\pi_N^M\circ\ga(t)}(\dot{\ga}_{(\ell,x)},\dot{\ga}_{(\ell,x)})=1+O(r_1).$ Thus, with $r_1$ small and $\e/r_1<\frac{1}{10}$ , we can deduce that
	\begin{align}\label{xlength}
		\hm^1(\ga)=(1+O(r_1))\int\sqrt{g_{\pi_N^M\circ\ga(t)}(\dot{\ga}_{(\ell,x)},\dot{\ga}_{(\ell,x)})}=(1+O(r_1))\hm^1(\pi_N^M\circ\ga).
	\end{align}
	Consider a point $p\in X'$ so that $\dis_{X'}(p,L_X')=\rh>0.$ Let $\ga_1$ be the unique geodesic from $\pi_N^M(p)$ to $L'$ in the intrinsic metric on $X'_N$. Define 
	\begin{align*}
		\ga=((\ga_1)_{(\ell,x)},0,\mu((\ga_1)_{(\ell,x)})).
	\end{align*}
	Then $\ga$ is a smooth curve on $X'$ from $p$ to some point on $L_X'.$  Equation (\ref{xlength}) gives
	\begin{align}\label{projdis}
		\dis_{X'}(p,L_X')\le \hm^1(\ga)=(1+O(r_1))\hm^1(\ga_1)=(1+O(r_1))\dis_{X'_N}(\pi_N^M(p),L').
	\end{align}The big $O$ depends only on the geometry of $N,X',L'.$
	\subsubsection{Distance to $X'$ under projection}
	On the other hand, let $q$ be a point in $U_j$ so that $\dis_M(p,q)=\dis_M(q,X').$ Let $\vp$ be the unique geodesic starting from $p$ to $q.$ Parameterize $\vp$ with $|\dot{\vp}|=1.$
	
	Note that $\dot{\vp}(0)\perp X'$ by first variation arguments. Let $v$ be a tangent vector to $X'$ with $|v|=1.$ Then by (\ref{eucnorm}) (\ref{xmet}) and $v_\nu=0$, we have
	\begin{align*}
		&g(v_{(\ell,x)},\dot{\vp}_{(\ell,x)}(0))=g(v,\dot{\vp}(0))-	g(v_{\mu},\dot{\vp}_{\mu}(0))-g(v_\mu,\dot{\vp}_{(\ell,x)}(0))-g(v_{(\ell,x)},\dot{\vp}_{\mu}(0))\\=&0+O(|\nu_\mu||\dot{\vp}_\mu(0)|)+O(r_1)|\nu_\mu||\dot{\vp}_{(\ell,x)}(0)|+O(r_1)|\nu_{(\ell,x)}||\dot{\vp}_\mu(0)|\\
		=&|\nu_{(\ell,x)}|(O(\e)|\dot{\vp}_\mu(0)|+O(\e r_1)|\dot{\vp}_{(\ell,x)}(0)|+O(r_1)|\dot{\vp}_\mu|(0)).
	\end{align*}
	Dividing both sides by $|v_{(l,x)}|$ and taking the maximum of the left side, we deduce that
	\begin{align}\label{vpnorm}
		|\dot{\vp}_{(\ell,x)}(0)|=O(r_1)|\dot{\vp}_\mu(0)|.
	\end{align}
	This in turn gives
	\begin{align*}
		|\dot{\vp}_{(\nu,\mu)}(0)|=1+O(r_1),|\dot{\vp}_{(\ell,x)}|=O(r_1).
	\end{align*}
	Note that this does not put any restriction on $|\dot{\vp}_\nu(0)|.$ It may be that $|\dot{\vp}_\nu(0)|\approx 1$ and the other components are all small.
	
	By (\ref{covder}), and $\e<r_1$,we deduce that
	\begin{align*}
		0=\na_{\dot{\vp}(t)}\dot{\vp}(t)=\ddot{\vp}(t)+O(|\vp|^2r_1),
	\end{align*}
	Since $\vp$ stay in $U_j$ by our assumption, we have
	\begin{align}\label{vpest}
		&\ddot{\vp}=O(r_1^3),\dot{\vp}(t)=\dot{\vp}(0)+O(r_1^3)t,\\\label{vpest1}&\vp(t)=p+t\dot{\vp}(0)+O(r_1^3)t^2,\\
		&\label{vpest2}|\dot{\vp}_{(\ell,x)}(t)|=O(r_1),|\dot{\vp}_{(\nu,\mu)}(t)|=1+O(r_1).
	\end{align}
	This implies that the parameter of the curve lies in  \begin{align}\label{timeest}
		[0,(1+O(r_1))\dis_M(p,q)].
	\end{align} Now we have
	\begin{align*}
		\hm^1(\vp)=&\int\sqrt{g_{\vp(t)}(\dot{\vp}(t),\dot{\vp}(t))}dt\\
		=&\int\sqrt{g_{\pi_N^M\circ\vp(t)}(\dot{\vp}(t),\dot{\vp}(t))+O(|\vp(t)|)}dt\\
		=&\int\sqrt{g_{\pi_N^M\circ\vp(t)}(\dot{\vp}(t)_{(\ell,x,\nu)},\dot{\vp}(t)_{(\ell,x,\nu)})+|\dot{\vp}_\mu(t)|^2+O(r_1)}dt
	\end{align*}
	If $|\dot{\vp}_\mu(t)|^2$ is larger than the $O(r_1)$ term in absolute value, then the above gives
	\begin{align*}
		\hm^1(\vp)\ge \hm^1(\pi_N^M\circ\vp).
	\end{align*}
	Alternatively, if $|\dot{\vp}_\mu(t)|^2$ is smaller than the $O(r_1)$ term. Then by (\ref{vpest}) (\ref{vpest1}) (\ref{vpest2}), we deduce that $|\dot{\vp}_\nu(t)|=1+O(\sqrt{r_1}).$ Then use (\ref{vpnorm}), we deduce that $|\dot{\vp}_{(\ell,x,\nu)}(t)|=1+O(\sqrt{r_1}).$ In this case we deduce that
	\begin{align*}
		\hm^1(\vp)=(1+O(\sqrt{r_1}))\hm^1(\pi_N^M\circ\vp).
	\end{align*}
	Combining both cases, we always have
	\begin{align}\label{lengthpq}
		\dis_{N}(\pi_N^M(q),X_N')\le\hm^1(\pi_N^M\circ\vp)\le(1+O(\sqrt{r_1}))\hm^1(\vp)=(1+O(\sqrt{r_1}))\dis_M(q,X').
	\end{align}
	\subsubsection{Projection of angled neighborhoods: proof of (\ref{goal})}
	By (\ref{vpest1}) and (\ref{timeest}), we have 
	\begin{align*}
		\pi_N^M(q)=&\pi_N^M(p+t_{max}\dot{\vp}(0)+O(r_1^3)t_{max}^2)=\pi_N^M(p)+t_{max}\pi_N^M(\dot{\vp}(0))+O(r_1^3)t_{max}^2\\
		=&\pi_N^M(p)+\dis_M(q,X')\pi_N^M(\dot{\vp}(0))+O(r_1\dis_M{(q,X')}).
	\end{align*}
	By (\ref{vpest2}) (\ref{lengthpq}) and (\ref{projdis}), this implies that
	\begin{align*}
		\pi_{X'_N}^N\circ\pi_N^M(q)=&\pi_{X'_N}^N\circ\pi_N^M(p)+\dis_M(q,X')\pi_{X'_N}^N\circ\pi_N^M(\dot{\vp}(0))+O(r_1\dis_M{(q,X')})\\
		=&\pi_N^M(p)+O(r_1)\dis_M{(q,X')}\\
		=&\pi_N^M(p)+O(r_1)\frac{\dis_M(q,X')}{\dis_{X'}(p,L_X')}\dis_{X'}(p,L_X')\\
		=&\pi_N^M(p)+O(r_1)\frac{\dis_M(q,X')}{\dis_{X'}(p,L_X')}\dis_{X'_N}(\pi_N^M(p),L').
	\end{align*}
	Now we are ready to prove (\ref{goal}). Note that by our assumptions, if $q$ belongs to $W_{(\frac{1}{2}r,\ta)}(X_N',L')$ for $r\le r_0,\ta\le \frac{\pi}{4},$ then $\frac{\dis_M(q,X')}{\dis_{X'}(p,L_X')}\le 1$ by definition. Thus, we have
	\begin{align}\label{projdiff}
		\pi_{X'_N}^N\circ\pi_N^M(q)=\pi_N^M(p)+O(r_1)\dis_{X'_N}(\pi_N^M(p),L')
	\end{align}
	Let $q_L$ be the point on $L'$ that is of nearest distance to $\pi_{X'_N}^N\circ\pi_N^M(q)$ with respect to the intrinsic distance on $X'_N.$ Then by the triangle inequality, (\ref{projdiff}) (\ref{lth}) and (\ref{projdis}), we have
	\begin{align}
		&\dis_{X_N'}(L',\pi_{X'_N}^N\circ\pi_N^M(q))=	\dis_{X_N'}(q_L,\pi_{X'_N}^N\circ\pi_N^M(q))\\\ge&\dis_{X_N'}(q_L,\pi_N^M(p)) -\dis_{X'_N}(\pi_N^M(p),\pi_{X'_N}^N\circ\pi_N^M(q))\\\ge &\dis_{X'_N}(\pi_N^M(p),L')-O(r_1)\dis_{X'_N}(\pi_N^M(p),L')\\
		=&(1+O(r_1))\dis_{X'_N}(\pi_N^M(p),L')\\
		\ge&(1+O(r_1))\dis_{X'}(p,L_X').\label{denom}
	\end{align}
	By combining (\ref{lengthpq}) and (\ref{denom}), we can deduce that
	\begin{align*}
		\frac{\dis_{N}(\pi_N^M(q),X_N')}{\dis_{X_N'}(L',\pi_{X'_N}^N\circ\pi_N^M(q))}\le \frac{(1+O(\sqrt{r_1}))\dis_M(q,X')}{(1+O(r_1))\dis_{X'}(\pi_{X'}^M(q),L_X')}.
	\end{align*}
	Considering (\ref{lth}), $\dis_{N}(0,\pi_N^M(q))\le (1+O(r_1))\dis_M(0,q)\le 2\dis_M(0,q)$, we have proven (\ref{goal}), with big $O$ depending continuously on the geometry of $X',N,M.$
	\subsubsection{Conclusion}
	If we repeat the argument above with $X'$ replaced by $Y'$, and combine the results together, we deduce that
	\begin{align}
		&\pi_N^M(W_{(\frac{1}{2}r,\ta)}^M(X',L'_X))\s W_{(r,\arctan[(1+O(\sqrt{r_1}))\tan\ta])}^N(X'_N,L'),\label{w1}\\
		&\pi_N^M(W_{(\frac{1}{2}r,\ta)}^M(Y',L'_Y))\s W_{(r,\arctan[(1+O(\sqrt{r_1}))\tan\ta])}^N(Y'_N,L'),\label{w2}
	\end{align}
	with big $O$ depending continuously on the geometry of $X',Y',N,M.$
	
	With $\e$ small, $r<r_1$ small and $\tau<\frac{\ta_0}{2}$, we always have	
	\begin{align}
		W_{(r,\tau)}^N(X'_N,L')\cap W_{(r,\tau)}^N(Y'_N,L')=L'.
		\label{wint}
	\end{align} This follows, for example, by considering (\ref{angnei}) and (\ref{eucnorm}).
	
	Thus, by (\ref{w1}) (\ref{w2}) and (\ref{wint}), we deduce that
	\begin{align}\label{finint}
		W_{(\frac{1}{2}r,\ta)}^M(X',L'_X)\cap W_{(\frac{1}{2}r,\ta)}^M(Y',L'_Y)\s(\pi_N^M)\m (L')
	\end{align}
	By (\ref{vpest1}) (\ref{vpest2}) and (\ref{timeest}), we know that any geodesic $\vp$ starting normal from $X'\setminus L_X'$  will not meet $(\pi_N^M)\m (L')$ in the coordinate patches as long as the end point $q$ satisfies $\dis_M(\pi_{X'}^M(q),q)\le \dis_{X'}(\pi_{X'}^M(q),L'_X)$, which is satisfied in $W_{(\frac{1}{2}r,\ta)}^M(X',L'_X)$, Thus, we deduce that $$W_{(\frac{1}{2}r,\ta)}^M(X',L'_X)\cap(\pi_N^M)\m (L')=L_X'.$$ Similar arguments for $Y'$ shows that
	$$W_{(\frac{1}{2}r,\ta)}^M(Y',L'_Y)\cap(\pi_N^M)\m (L')=L_Y'.$$ By (\ref{finint}), we deduce that
	\begin{align*}
		W_{(\frac{1}{2}r,\ta)}^M(X',L'_X)\cap W_{(\frac{1}{2}r,\ta)}^M(Y',L'_Y)\s L_X'\cap L_Y'.
	\end{align*}
	We are done.
	\subsection{Integral currents and calibrations}
	Since we mostly work on the geometric side, the reader can just assume integral currents are chains over immersed oriented submanifolds with singularities. The action of an integral current on a differential form is integration on the chain. If the chain is an oriented smooth submanifold $\Si$ with mild singularities (possibly with boundary), we will usually abuse the notation and write $\Si$ instead of the usual literature convention $\cu{\Si}$.
	
	For calibrations, the reader needs to know the definition of comass and calibrations (Section II.3 and II.4 in \cite{HLg}). The primary reference is \cite{HLg}. The most important thing to keep in mind is the fundamental theorem of calibrations (Theorem 4.2 in \cite{HLg}), i.e., calibrated currents are area-minimizing among homologous competitors.  We will use it many times without explicitly citing it.
	
	The following lemma is folklore, but we could not find a reference for it. We provide quick proof here.
	\begin{lem}\label{hodge}
		Let $\ast$ denote the Hodge star on the Euclidean space $\R^{m+n}$. Then for any constant  $m$-form $\phi$ on , we have
		\begin{align*}
			\cms\phi=\cms\ast\phi.
		\end{align*}
	\end{lem}
	\begin{proof}
		Every such $\phi$ can be written as a finite sum
		\begin{align*}
			\phi=\sum_j a_jX_j\du,
		\end{align*}with $a_j$ nonzero  real numbers and $X_j$ oriented $m$-dimensional coordinate planes.
		
		Let $Y$ be any oriented $m$-plane, and $Y^\perp$ be its orthogonal complement, oriented so that 
		\begin{align*}
			Y\w Y^\perp= \R^{m+n}.
		\end{align*} 
		This implies that
		\begin{align*}
			(\R^{m+n})\du=Y\du\w (Y^\perp)\du.
		\end{align*}
		If use an orthonormal coordinate adapted to $Y,Y^\perp,$ then coordinate calculation yields
		\begin{align*}
			\ast Y\du=(Y^\perp)\du.
		\end{align*}
		Then we have
		\begin{align*}
			\ast\phi=\sum_j a_j(Y^\perp)\du,
		\end{align*}
		by definition of Hodge star. Thus, by the determinant definition of simple forms, we have
		\begin{align*}
			\phi(Y)=&\sum_j a_jX_j\du(Y)=\sum_j a_j\ri{X_j\du,Y\du},\\
			\ast\phi(Y)=&\sum_j a_j(X_j^\perp)\du(Y^\perp)=\sum_j a_j\ri{(X_j^\perp)\du,(Y^\perp)\du}=\sum_j a_j \ri{\ast X_j\du,\ast Y\du}\\
			=&\sum_j a_j \ri{\ast X_j\du\w \ast\ast Y\du, (\R^{m+n})\du}=\sum_j a_j(-1)^{mn}\ri{\ast X_j\du\w Y\du,(\R^{m+n})\du}\\
			=&\sum_j a_j\ri{Y\du\w\ast X_j\du,(\R^{m+n})\du}=\sum_j a_j\ri{Y\du,\ast\ast X_j\du}\\
			=&(-1)^{mn}\sum_j a_j\ri{X_j\du,Y\du}\\
			=&(-1)^{mn}\phi(Y).
		\end{align*}
		By definition of comass, the result is immediate.
	\end{proof}
	The following lemma will be useful.
	\begin{lem}\label{prodcal}
		Endow the standard $\R^m\times \R^n$ with the standard metric. Let $A_1,\cd,A_k$ be a finite collection of distinct $d_1$-dimensional linear subspaces in $\R^m$ in general position with $d_1\ge 2,k\ge 2.$  Let $B_1,\cd,B_k$ be a finite collection of not necessarily distinct $d_2$-dimensional linear subspaces of $\R^n$, with $d_2\ge 0,k\ge 2.$ If $$\Ta(A_i,A_j)=\frac{\pi}{2}$$ for $i\not=j$, i.e., pairwise orthogonal modulo intersection, and intersection being of codimension at least two, then the following form $\phi$
		\begin{align*}
			\phi=\sum_j c_jA_j\du\w B_j\du,
		\end{align*} with $c_j>0$ constants, has comass
		\begin{align*}
			\cms\phi=\max_j|c_j|,
		\end{align*}  and $A_j\times B_j$ is calibrated by $\phi$ if $\cms\phi\le 1,c_j=\pm 1.$
	\end{lem} 
	\begin{rem}
		Here $d_2=0$ means $B_j\du=1$.
	\end{rem}
	\begin{proof}
		By Lemma \ref{hodge}, it suffices to show that $\cms \ast\phi=1.$ Dimension counting gives that the orthogonal complement of $A_j\oplus B_j$ is precisely $A_j^\perp\oplus B_j^\perp.$ Thus, we have
		\begin{align*}
			\ast\phi=&\sum_j c_j\ast(A_j\du\w B_j\du)\\
			=&\sum_jc_j ((A_j\oplus B_j)^\perp)\du\\
			=&\sum_j c_j(A_j^\perp)\du\w (B_j^\perp)\du.
		\end{align*}
		Note that the collection $A_j^\perp$ consists of pairwise perpendicular and linearly independent subspaces. Thus, an inductive application of Lemma 2.1 in \cite{DHM} does the job for us.
	\end{proof}
	\begin{rem}
		It may very well happen that $\phi$ calibrates more than just $A_j\times B_j$. For example, the case of $d_1=2$, $d_2=0$,
		$\phi$ can be Kahler forms. 
	\end{rem}
	\subsection{Gluing of calibrations}
	In this subsection, we will give some useful lemma about gluing calibrations. 
	\begin{lem}Let $g,h$ be smooth metrics on a manifold, and $\phi$ a smooth form. Then we have\label{prodcom}
		\begin{align*}
			\cms_{g^{a}h^b}\phi\le (\cms_g\phi)^a(\cms_h\phi)^b
		\end{align*}
	\end{lem}
	\begin{proof}Let $V$ denote a simple vector. We have 
		\begin{align*}
			\no{V}_{g^ah^b}=\no{V}_{g^a}\no{V}_{h^b}=(\no{V}_g)^a(\no{V}_h)^b.
		\end{align*}
		By Remark 2.2 in \cite{YZa}, we have
		\begin{align*}
			\cms_{g^ah^b}(\phi)=&1/\min\{\no{V}_{g^ab^b}|\phi(V)=1\}\\
			=&
			1/\min\{(\no{V}_g)^a(\no{V}_g)^b|\phi(V)=1\}\\
			\le&\bigg(\frac{1}{\min\{\no{V}_g|\phi(V)=1\}}\bigg)^a
			\bigg(	\frac{1}{\min\{\no{V}_h^b|\phi(V)=1\}}\bigg)^b\\
			=&(\cms_g\phi)^a(\cms_h\phi)^b.
		\end{align*}
	\end{proof}
	Next, we are ready to state our main lemma. Suppose we have the following setup.
	\begin{lem}\label{glucalfor}
		\begin{itemize}
			\item  $M$ is a compact $m+n$-dimensional smooth (not necessarily orientable)  manifold (with possibly nonempty boundary), with $m\ge 1, n\ge 1$.
			\item  $\phi$ is a smooth calibration $m$-form in a smooth metric $g$ on $M.$
			\item There exist a smooth open subset $O\s M,$ and a sequence of smooth calibration $m$-forms $\{\phi_j\}$ in a sequence of smooth metric $\{g_j\}$, both defined on $O.$
			\item $g_j\to g$ and $\phi_j\to \phi$ smoothly in $O.$
			\item  $\phi_j-\phi$ is an exact form in $O.$
			\item  $B_\e^-(\pd O)$ is a one sided tubular neighborhood of $\pd O$ inside $O$, so that the exponential map in the normal bundle of $\pd O$ maps injectively onto $B_\e^-(\pd O)$.  
		\end{itemize}
		With the above assumptions, we have the following constructions.
		\begin{itemize}
			\item There exists a sequence of smooth metrics $\ov{g_j}$ and a sequence of smooth calibration $m$-forms $\ov{\phi_j}$ both defined globally on $O.$
			\item $\ov{g_j}\to g$ and $\ov{\phi_j}\to \phi$ smoothly.
			\item We have the following restriction properties
			\begin{align*}
				\ov{g_j}|_{B_\e^-(\pd O)\cp\cap O}=g_j,\ov{g_j}|_{ O\cp}=g,\\\ov{\phi_j}|_{B_\e^-(\pd O)\cp\cap O}=\phi_j,\ov{\phi_j}|_{O\cp}=\phi.
			\end{align*}
		\end{itemize}
		\begin{proof}
			The proof is a straightforward bridging of $(\phi,g)$ to $(\phi_j,g_j)$ in $B_\e^-(\pd O)$. Let $\rh$ denotes the signed distance function to $\pd O$. Calculation in Fermi coordinates verifies that $\rh$ is smooth. Let $\ta_{(a,b)}$ be a smooth transition function with positive real parameters $a,b$, so that $\tau_{(a,b)}$ that is monotonically decreasing on $\R,$, $$\tau|_{[-(b-a),\infty)}=0,\tau|_{(-(b+a),-1]}=1.$$ By Theorem 3.1.1 in \cite{GS}, there exists a smooth form $\psi_j$ defined on $O$, so that
			\begin{align*}
				d\psi_j=\phi-\psi_j,
			\end{align*} and $\psi_j\to 0$ smoothly. (First, by naturality of exterior derivative, we know that $\psi_j-\phi$ is exact on $\pd B_\e^-(\pd O)$. Apply Theorem 3.1.1 with no boundary, i.e., zero boundary condition, we get an antiderivative of $\psi_j-\phi$ on $\pd B_\e^-(\pd O)$ which converges smoothly to $0$ as $j\to\infty$ by estimate (1.3) on page 114 of \cite{GS}. Then use this antiderivative as the boundary condition and apply Theorem 3.1.1 again.) 
			
			Define
			\begin{align*}
				\ov{\phi_j}=
				\begin{cases}
					\phi_j,&\textnormal{ in }O\cap B^-_\e(\pd O)\cp,\\
					\phi_j+d\bigg(\tau_{(\frac{\e}{2},\frac{\e}{4})}(\rh)\psi_j\bigg), &\textnormal{ in }B^-_\e (\pd O),\\
					\phi , &\textnormal{ in }O\cp.
				\end{cases}
			\end{align*}
			Note that by construction, for points of distance more than than $\frac{3\e}{4}$ to  $\pd O$ inside $O,$ $\ov{\phi_j}$ equals $\phi_j.$ Similarly for points of distance less than than $\frac{\e}{4}$ to  $\pd O$ in $O,$ $\ov{\phi_j}$ equals $\phi.$
			
			Thus $\ov{\phi_j}$ is smooth. It is straightforward to verify that $d\ov{\phi_j}=0,$ $\ov{\phi_j}\to \phi$ and $\ov{\phi_j}$ has the restriction properties as claimed.
			
			Let $\mu_j$ denote maximum of the comass of $\ov{\psi_j}$ in $g_j$ and in $g$ in $B_\e^-(O).$ Note that $\mu_j\to 1$ since $\psi_j\to 0,\phi_j\to \phi,g_j\to g$ smoothly. 
			
			Let $\be_{(b,a)}$ be a smooth bump function with positive real parameters $a,b,$ so that $$0\le \be_{(b,a)}\le 1,\be_{(b,a)}|_{[-(b+a),-(b-a)]}=1,\be_{(b,a)}|_{(-\infty,-(b+\frac{4}{3}a)]}=\be_{(b,a)}|_{[-(b-\frac{3}{2}a,\infty)}=0.$$ 
			Define
			\begin{align*}
				\ov{g_j}=
				\begin{cases}
					g_j,&\textnormal{ in }O\cap B^-_\e(\pd O)\cp,\\\bigg(\beta_{\frac{\e}{2},\frac{\e}{3}}(\rh)\bigg(\mu_j^{\frac{2}{m}}-1\bigg)+1\bigg)g_j^{\tau_{(\e/2,\e/3)}(\rh)}g^	{1-\tau_{(\e/2,\e/3)}(\rh)}
					, &\textnormal{ in }B^-_\e (\pd O),\\
					g, &\textnormal{ in }O\cp.
			\end{cases}\end{align*}
			Note that $\tau$ terms serve to transit between $g$ and $g_j$ for $\rh\in (-\frac{5}{6}\e,-\frac{1}{6}\e).$ The $\be$ factor transits between $\mu_j^{\frac{2}{m}}$ and $2$ for $\rh\in (-\frac{17}{19}\e,-\frac{5}{6}\e)\cup (-\frac{1}{6}\e,-\frac{1}{18}\e)$. Thus, the three pieces of $\ov{g_j}$ comes together smoothly. Using similar interval by interval checking, it is straightforward to verify that $\ov{\phi_j}$ is a calibration in $\ov{g_j}$ using Lemma 3.1 in \cite{YZa} and Lemma \ref{prodcom}.  
		\end{proof}
	\end{lem}
	The following lemma will be useful for determining the exactness condition above.
	\begin{lem}
		\begin{itemize}
			\item $U$ is a smooth domain in a smooth manifold $M^{m+n}.$
			\item $H_m(U,\Z)=\sum_j\Z[T_j]$ is finitely generated by integral currents $T_j.$
			\item The classes $[T_j]\s H_m(M,\Z)$ are linearly independent over $\Z.$
		\end{itemize}
		Suppose we have a closed differential form $\phi$ in $U,$ with $T_j(\phi)=\ai_j\not=0.$ Then any closed forms $\psi$ in $M$ with $T_j(\psi)=\ai_j$ lies in the same de Rham cohomology class as $\phi$ in $U.$
	\end{lem}
	\begin{proof}
		Since $H_m(U,\Z)=\sum_j\Z[T_j],$ by the universal coefficient theorem, we deduce that $H^m(U,\R)\cong\hom(H_m(U,\R),\R)\cong\hom(H_m(U,\Z)\ts\R,\R).$	By the linear independence of $T_j$ in $H_m(M,\Z)$, we know that $T$ form a basis of $H_m(U,\Z)\ts \R.$ The lemma follows directly from regarding $\psi$ and $\phi$ as linear functionals on $H_m(U,\Z)\ts \R.$
	\end{proof}
	\subsection{Definition of mod $v$ minimizing}
	The definition of mod $v$ minimizing is the same one as Definition 1.2 in \cite{DPHM}. \begin{defn}\label{modu}
		An integral current $T$ on a manifold $M$ is a cycle (or boundary) mod $v,$ if for some integral current $S,$ $T+vS$ is a cycle (or boundary, respectively) as an integral current. We say a mod $v$ cycle $T$  is area minimizing mod $v,$ if for any mod $v$ boundary $S,$ we have
		\begin{align*}
			\cm(T)\le\cm(T+S),
		\end{align*}with $\cm$ the mass of currents.
	\end{defn}
	\begin{rem}
		By Corollary 1.5 and 1.6 in \cite{RY}, the above definition is equivalent to the canonical definition of mod $v$ flat chains. 
	\end{rem}	
	\section{Proof of Theorem \ref{prevalence}, \ref{prev} and Theorem \ref{tor}}
	\subsection{Proof of Theorem \ref{prev}}
	\subsubsection{Contents of Theorem 1.2 in \cite{ZL}}\label{recap}First, let us recall the proof of Theorem 1.2 in \cite{ZL}. Take the standard torus $\TO^{2(n-j)}$ and any $j$-dimensional manifold $N_j.$ Split the torus into product $\TO^{2(n-j)}=\TO^{n-j}\times \TO^{n-j}$ and use $(x_1,\cd,x_{n-j})$ and $(y_1,\cd,y_{n-j})$ to denote the coordinates lifting the factors to $\R^{n-j}.$ Define
	\begin{align}
		\cu{x_1\cd x_{n-j}}=\TO^{n-j}\times\{0\},\cu{y_1\cd y_{n-j}}=\{0\}\times \TO^{n-j}.
	\end{align}
	Then the two currents intersect only at $\{0\}\times\{0\},$ an isolated point. Let $f_j$ be a smooth function on $N_j,$ a $j$-dimensional closed smooth manifold. We can make the $C^\infty$ norm arbitrarily small by multiplying $f_j$ with small constants. Abusing the notation a bit, we will also use $f_j$ to denote the graph of $f_j$ in $N_j\times[-\frac{1}{4},\frac{1}{4}]\s N_j\times S^1.$
	
	Then we construct in Section 3 of \cite{ZL} a smooth metric $g$ and calibration form $\phi$ on 
	\begin{align}\label{o1}
		\TO^{2(n-j)}\times O_1,
	\end{align} with $O_1$ a tubular neighborhood of $f_j$ in $N_j\times S^1$ which contains $N_j$, so that
	\begin{align}\label{tsum}
		T=\cu{x_1\cd x_{n-j}}\times N_j+\cu{y_1\cd y_{n-j}}\times f_j
	\end{align}
	is calibrated by $\phi$ in $g.$ The singular set of $T $ is precisely $f_j\cap N_j,$ which is prescribed closed subset $K_j$ of $N_j.$
	\subsubsection{Building blocks of singularities}
	Split $\R^4$ into $\R^2\oplus\R^2$. Similar to the previous subsection, denotes the first subspace, the $x_1x_2$-plane, by $X$, and the second factor, the $y_1y_2$-plane, by $Y$, . Let 
	\begin{align}
		&\si(Y)=(Y\times \R^{n-1})\cap S^{n+2},\si(X)=(X\times \R^{n-1})\cap S^{n+2}\\
		&C=X+Y,\si(C)=(C\times \R^{n-1})\cap S^{n+2},
	\end{align} 
	By Lemma 4.1 in \cite{ZL1}, there exists a tubular neighborhood $U$ around the singular set of $\si(C)$ and diffeomophism $\Gamma$ from $U$ to $B_1^4\times S^{n-2}$ so that \begin{align}\label{uu}
		\Ga(U)=B_1^4\times S^{n-2},\Ga(\si(C)\cap U)=[(X+Y)\times S^{n-2}]\cap\Ga U.
	\end{align} 
	Now in the notation of the last subsubsection, regard $[-\frac{1}{2},\frac{1}{2}]\s S^1$, set $N_{n-2}=S^{n-2}$ and let $f_{n-2}$ be the a smooth function with a prescribed zero set $K\s N_{n-2}$ and small $C^\infty$ norm. Consider $(\Ga\times\id)\m(Y\times f_{n-2})$ in $U\times [-\frac{1}{2},\frac{1}{2}].$ With $C^\infty$ norm of $f_{n-2}$ small enough, we can assume that $(\Ga\times\id)\m(Y\times f_{n-2})$ is obtained from exponentiating a section $v_U$ in its normal bundle of  $\si(Y)\cap U\times [-\frac{1}{2},\frac{1}{2}]$. Moreover the $C^\infty$ norm of $v_U$ is controlled by that of $f_{n-2}.$   
	Note that $\si(Y)$ is a totally geodesic $n$-sphere of $S^{n+2},$ thus having trivial normal bundle in the inclusion $\si(Y)\s S^{n+2}\s S^{n+2}\times S^1.$ Thus, we can easily extend $v_U$ to a section $$v_f$$ in the normal bundle of $\si(Y)$ with $C^\infty$ norm controlled by that of $f_{n-2}$ and $v_f$ equal to $0$ on $\si(Y)\cap B_{\e}(U)\cp,$ i.e., portion of $\si(Y)$ of distance $\e$ away from $U$. Define
	\begin{align*}
		Y_f=\exp^\perp_{\si(Y)\s S^{n+2}\times S^1}v_f.
	\end{align*}
	It is straightforward to verify that with $\e,\no{f_{n-2}}_{C^\infty}$ small, one can ensure that $\si(X)$ intersects $Y_f$ only inside $U,$ and the intersection is precisely the zero set $K$ of $f_{n-2}$ on the intersection $\si(X)\cap\si(Y)=S^{n-2}.$
	
	By the discussion in the last section, we know that there exists a smooth metric $g$ and calibration form $\phi$ in $$U\times [-\frac{1}{2},\frac{1}{2}],$$ so that $\si(X) +Y_f$ is calibrated by $\phi$ in $g.$ (To see this, note that by construction, $\si(X)+Y_f$ is sent to (\ref{tsum}) by $\Ga$.)
	
	Define a map $I:S^{n+2}\times [-\frac{1}{2},\frac{1}{2}]\s B_2^{m+n},$ by 
	\begin{align}
		I(p,t)\mapsto ((1+t)p,0,\cd,0),
	\end{align}
	where $p\in S^{n+2}\s \R^{n+3},$ and the last $m-3$ coordinates are zero. It is straightforward to verify that $I$ is an embedding.
	
	Now equip $\R^{m-3}$ with the standard flat metric $\de$, then it is straightforward to verify that in $I(U\times[-\frac{1}{2},\frac{1}{2}])\times \R^{m-3},$ the form $I\pf(\phi)$ extended constantly along $\R^{m-3}$ directions is a calibration form in the product metric of $I\pf g$ and $\de$. Call this extended form $\ov{f}$ and the product metric $\ov{g}.$ 
	
	To sum it up, we have proven that there exists a smooth metric $\ov{g}$ and a smooth form $\ov{\phi}$ defined in $I(U\times[-\frac{1}{2},\frac{1}{2}])\times \R^{m-3}$, that calibrates the restriction of $I\pf(\si(X)+Y_f)$ in that region. 
	\subsubsection{Finishing the proof}
	Let $\Si$ be a smoothly embedded representative of the homology class we consider. Pick a point $q\not\in \Si$ and consider the ball around $q.$ Without loss of generality, we can identify it as $B_2^{m+n}\s \R^{m+n}.$ By our previous discussion, $\si(X)+Y_f$ supported in $S^{n+2}\times S^1$ embeds into the intersection of $I(U\times [-\frac{1}{2},\frac{1}{2}])\times \R^{m-3}$ and $B_2^{m+n}.$ Pick three points $p\in \Si,a\in I\pf(\si(X)),b\in I\pf(Y_f) $, with $a,b\not\in I(U\times [-\frac{1}{2},\frac{1}{2}]).$ Then do a connected sum of $\Si,\si(X)$ and $Y_f$ using two small necks from $p$ to $a$ and $a$ to $b.$ Since the codimension $n$ is at least $3,$ by transversality, we can assume that the resulting $\Si\# \si(X)\# Y_f$ is immersion, and the self-intersection set is still $\si(X)\cap Y_f.$ By transversality and making $\e,\no{f_{n-2}}_{C^\infty}$ small, we can assume the existence of a small open set $V,$ so that $(\si(X)\cap\si(Y))\cup(\si(X)\cap Y_f)\s V,$ and in $V$, $\Si\# \si(X)\# Y_f$ equals $\si(X)+Y_f.$ Now apply Lemma 2.7 in \cite{ZL1}. We are done.\subsection{Proof of Theorem \ref{prevalence}}
	By \cite{RT}, for any integral homology class, some multiple of it can be represented by embedded smooth orientable submanifolds. Say the $d$-dimensional rational homology of $M^{d+c}$ is nontrivial. Then we know that some $d$-dimensional integral homology class of $M$ can be represented by a smoothly embedded orientable submanifold. Just apply Theorem \ref{prev}.
	\subsection{Proof of Theorem \ref{tor}}
	By the assumptions, we can decompose $\TO^{m+n}$ as a product
	\begin{align*}
		\TO^{m+n}=\TO^{2(m-l)}\times \TO^{l}\times \TO^1\times\TO^{n-m+l-1}.
	\end{align*}
	To see this, consider the span of $\TO_1$ and $\TO_2,$ which is a $\TO^{2m-l}.$ Since the intersection is not transverse, we know that $m+m-l< m+n,$ which gives $1\le n-m+l.$ Thus, we can always split as product $\TO^{m+n}=\TO^{2m-l}\times \TO^{n-m+l}$ and consider $\TO_1,\TO_2$ sitting inside $\TO^{2m-l}\times\{0\}.$ We do a further splitting off the intersection $\TO_3$ to get $\TO^{2m-l}=\TO^{2(m-l)}\times\TO_3=\TO^{2(m-l)}\times \TO^l.$ We also split a further $\TO^1$ out of $\TO^{n-m+l}.$
	
	For the following, we always work in the first three factors $\TO^{2(m-l)}\times \TO^{l}\times \TO^1$ until the last step. By the recap in Section \ref{recap}, for any closed set of $K,$ we can find a smooth perturbation of $\TO_2,$ $\TO_f=\cu{y_1\cd y_{m-l}}\times f_l,$ where $f_l$ has $K$ as zero set, and a smooth metric and a calibration that calibrates the sum of $\TO_1$ If we set $f^j_l=j\m f_l,$ and consider the construction obtained by substituting $f_l$ with $f^j_l,$ then we get a sequence of perturbations of $\TO_1+\TO_2,$ calibrated in $\TO^{2(m-l)}\times T^l\times [-\frac{1}{3},\frac{1}{3}]$ by a sequence of metrics $g_j$ and forms $\phi_j$ that converges to the standard flat metric $\de$ and the form $\TO_1\du+\TO_2\du$.  Now just apply Lemma \ref{glucalfor} with $g=\de,\phi=\TO_1\du+\TO_2\du$ and $g_j,\phi_j$ the sequence we constructed. If $n-m+l-1=0,$ then we are done. If $n-m+l-1>0,$ then we just use the product metric of the metric we obtained on $\TO^{2(m-l)}\times \TO^l\times \TO^1$ with the flat metric on $\TO^{n-m+l-1}$ and extend the forms constantly along the $\TO^{n-m+l-1}$ directions. The conclusion still holds.
	\section{Proof of Theorem \ref{mod}, Theorem \ref{modl}, Corollary \ref{cte}}
	\subsection{Proof of Theorem \ref{modl}}
	The proof is a direct extension of the proof of Theorem 1.2 in \cite{ZL1}. 
	\subsubsection{Recap of Section 3 in \cite{ZL1}} Let us recall the proof of Theorem 1.2 in \cite{ZL1}, in which case all self-intersections are transverse. 
	
	First, by Allard type regularity, we know that any nearby area-minimizing current $T'$ in nearby metrics is smooth and graphical over $T$ outside of a neighborhood $U$ of the transverse intersection set of $T.$ Moreover, $U$ shrinks to the self-intersection set as the current $T'$ converges to $T.$
	
	Without loss of generality, assume $T'$ and $T$ has smooth boundary in each connected component of $U.$ By our assumptions, in each connected component of $U,$ $\pd(T'\res U)$ has two connected components $V_1,V_2$ corresponding to two sheets of $T'$ in each component of $U.$ Similarly for $T.$ Denote the two sheets of $T$ in each component by $X,Y.$
	
	If we solve for the area-minimizing $X',Y'$ current with boundary $V_1,$ (and $V_2$) to get $\pd X'=V_1,\pd Y'=V_2.$ Then by Allard type regularity, we know that $X',Y'$ are smooth and graphical over $X,Y$, respectively. Thus, $X'\cap Y'=L'$ is a submanifold coming from the transverse intersection.
	
	The key result is Lemma 3.5 in \cite{ZL1}. We construct a vanishing calibration $\Psi_X'$  that depends geometrically on $X',L',$ so that $\Psi_Y'$ vanishes outside of $W_{(r,\ta)}(X',L')$. Similarly for $Y'.$ Here $\ta<2\arctan \frac{2}{\sqrt{a^2-4}},$ with $a$ the codimension of $L'$ in $X'$, and $r$ can be estimated to have uniform lower bound depending on $\ta,X,Y.$ With $r$ small, we can ensure that $W_{(r,\ta)}(X',L')\cap W_{(r,\ta)}(Y',L')=L',$ and thus $\Psi_X'+\Psi_Y'$ can serve as a calibration form that calibrates $X'+Y'$ by Lemma 2.2 in \cite{ZL1}. 
	
	However, the interior of homological area-minimizing currents with multiplicity $1$ has a unique continuation (Lemma 2.11 in \cite{ZL1}). Thus $X'+Y'$ is precisely $T\res U$ and we are done.
	\subsubsection{Modifications for our proof}
	The biggest difference here is that the two components $X,Y\s U$ no longer intersect transversely. The intersection is merely clean and thus we have no way to determine the intersection of nearby $X',Y'.$ Fortunately, Lemma \ref{ang} provides $L_X',L_Y'$, so that $W_{(r,\ta)}(X',L'_X)\cap W_{(r,\ta)}(Y',L'_Y)=L'_X\cap L_Y'.$ Thus if we use the vanishing calibration form above with $L_X'$ replacing $L'$ in $X'$ and similarly for $Y'$, we can still get a vanishing calibration $\Psi_X'+\Psi_Y'$ that calibrates $X'+Y'$. The argument goes through and we are done. In the mod $v$ case, just replace the calibration form with retractions as in Section 5 of \cite{ZL}. The result about smooth convergence follows from \cite{WA} and \cite{WA2}.
	\subsection{Proof of Theorem \ref{mod}}
	This is a direct corollary of Theorem \ref{modl}, provided the current we construct can be made arbitrarily close to an area-minimizing current with a clean intersection. Notice that we always have an auxiliary function determining the intersection set. We can just choose a sequence of numbers $\e_j\to 0,$ and multiply the auxiliary functions. Then use Lemma 2.7 in \cite{ZL1}. It is straightforward to verify that when the auxiliary function is precisely zero we have clean intersections only. When $\e_j$ is small, then Theorem \ref{modl} applies and we are done. In the mod $2$ case, just replace the calibration form with retractions as in Section 5 of \cite{ZL}.
	\subsection{Proof of Corollary \ref{cte}}
	By Theorem \ref{mod}, we know that moduli space of area-minimizing currents in $U$ in for metrics in an open set $\Om_0$ consists of immersions of the diffeomorphic manifolds.  
	
	To prove the corollary, the goal is to find a sequence of metric $g_j\to g$ and $g_j\in \Om_0,$a sequence of area-minimizing currents $T_j\to T,$ so that each $T_j$ is the unique minimizer and each $T_j$ has nontrivial transverse intersections. (We do not assume this is the only kind of singularity of $T_j$) By uniquely minimizing and Theorem \ref{mod}, for each $j$ there exists an open set $\Om_j\s \Om_0$ in the space of metrics, so that all area-minimizers $T_j'$ will be immersions that are close to $T_j.$ By transversality and making $\Om_j$ small, we can ensure that $T_j'$ also have non-empty transverse intersections. Then $\Om_T=\cup_{j\ge 1}\Om_j$ does the job for us. The generic case follows directly from Theorem 1 in \cite{BW2}.
	
	To find such a sequence, note that we always have an auxiliary function $f$ determining the intersection set, and our calibrations and metrics around the singular set depend smoothly on $f.$ Let $p$ be a point with $f(p)=0.$ Let $\be$ be a smooth bump supported near $p$ in a normal coordinate system $(x_1,\cd,x_n)$. Then $f_j=f+j\m\be x_1$ converges smoothly to $f$ as $j\to\infty,$ and the graph of $f_j$ intersects the domain transversely near $p.$ It is straightforward to verify that the currents constructed with $f_j$ as auxiliary functions have at least one transverse intersection point. Now we just apply Lemma 2.7 in \cite{ZL1} and we are done.
	\section{Proof of Theorem \ref{torus}}
	\subsection{Setting up the coordinate}
	Let $\R^{d+c}$ be the standard Euclidean space that is the universal cover of $\TO^{m+n}.$  By our assumptions, there exist subspaces $X_1,\cd,X_k$ in general position that projects down to $\TO_1,\cd,\TO_n.$ 
	
	Note that there exist a standard coordinate system $(x_1,\cd,x_{d+c})$ on $\R^{d+c}$, os that each $X_j$ is spanned by coordinate vector fields $x_{j_1},\cd,x_{j_d}$. To see this, set up an arbitrary orthonormal basis of $\R^{m+n}.$ Consider the projection matrices $P_j$ defined by orthogonal projection into $X_j.$ It is straightforward to verify that $P_jP_i=P_iP_j$ (due to orthogonality of the planes) and $P_j^\dagger=P_j$ (direct consequence of orthogonal projection). Thus, we have a finite set of commuting Hermitian matrices with real entries. Then spectral theorem and simultaneous diagonlizability results give the orthonormal basis $\{x_l\}$. 
	\subsection{The calibration form}
	Since constant forms projects down to well-defined smooth forms on $\TO^{m+n},$ we deduce that $X_j\du=dx_{j_1}\w\cd\w dx_{j_d}$ calibrates $\TO_j.$ 
	
	Equip $N$ with an arbitrary smooth metric $$h_N$$. Without loss of generality, we can assume that there exists a smooth domain $U$ containing $N_1,\cd,N_k,$ so that the normal bundle exponential map of each $N_j$ maps injectively onto $U.$ Let $\om_j$ be the volume form of $N_j$ and $\pi_j$ the nearest distance projection to $N_j.$ Define
	\begin{align*}
		\phi=\sum_jX_j\du\w\pi_j\du\om_j.
	\end{align*}
	By construction, $d\phi=0$ (since each $X_j\du$ is constant and $d\pi_j\du\om_j=\pi_j\du(d\om_j)=0.$)
	\subsection{Prototype metric} 
	We can write the metric on $\TO^{d+c}$ as
	\begin{align*}
		\de=\sum_j dx_j^2.
	\end{align*}
	Our aim is to find smooth functions $\la_1,\cd,\la_{d+c},$ so that in the metric
	\begin{align}\label{metric}
		g=\sum_j \lam_j dx_j^2+h_N.
	\end{align}
	At any point $p$
	\begin{align*}
		e_1=	\lam\m_1\pd_1,\cd,e_{d+c}=\lam\m_{d+c}\pd_{d+c},\nu_1,\cd,\nu_{m+n}
	\end{align*}
	forms an orthonormal basis in $g,$ where $\pd_j$ is the coordinate vector associated with $x_j,$ $\nu_1,\cd,\nu_{m+n}$ is an arbitrary orthonormal basis of $\Pi_N(p),$ with $\Pi_N$ the projection on to the $N$ factor. The dual basis is
	\begin{align*}
		e_1\du=	\lam_1dx_1,\cd,e_{d+c}\du=\lam_{d+c}dx_{d+c},\nu_1',\cd,\nu_{m+n}',
	\end{align*}
	where $\nu_1',\cd,\nu_{m+n}'$ is the dual basis to $\nu_1,\cd,\nu_{m+n}$ in $N.$
	Thus we can write 
	\begin{align}\label{cali}
		\phi=\sum_{j=1}^k (\cms_{N}\pi\du_j\om_j\lam\m_{j_1}\cd\lam\m_{j_d})e_{j_1}\du\w\cd\w e_{j_d}\du\w\frac{\pi\du_j\om_j}{\cms_{N}\pi\du_j\om_j}.
	\end{align}
	Since $\pi_j\du\om_j$ is a simple form, its comass is just its norm, i.e.,
	\begin{align*}
		\cms_N\pi_j\du\om_j=\sqrt{h_N(\pi_j\du\om_j,\pi_j\du\om_j)}.
	\end{align*}
	Thus, we can write
	\begin{align*}
		\frac{\pi\du_j\om_j}{\cms_{N}\pi\du_j\om_j}=B_j\du,
	\end{align*}where  $B_j$ is a smooth section of $n$-dimensional planes in $U\s N.$ Thus, if the coefficient for every term in (\ref{cali}) is precisely $1,$ then Lemma \ref{prodcal} would imply that $\cms_g\phi=1.$ Thus, we need to solve the system formed by simultaneously for all $j,$ \begin{equation*}
		\sum_{l=1}^d\log\lam_{j_l}=\log\cms_N\pi_j\du\om_j.
	\end{equation*}
	Note that if we regard $\log\lam_{j_l}$ as unknowns, then this is a linear system. The $j$-th row of the corresponding matrix has $0$ in the coordinate direction not in $X_j$ and $1$ in the coordinate directions contained in $X_j.$ We can explicitly reduce the system to a determined system as follows. First, one can take out all rows with $1$ only, then take out all rows with $0$ only. (This corresponds to setting $\lam_l=1$ for $\pd_l$ either contained in all $X_j$ or none of $X_j.$ Note that for $k\ge 2,$ all $\pd_l$ is contained in at least one $X_j.$) Then every column left has one and only one $0$, with $1$ in the rest entries, since the orthogonal complements of $X_j$ are linearly independent. Note that for each column, there exist precisely $c$ columns equal to it, again by linear independence of orthogonal complements to $X_j.$ Then one can replace all equal columns with a single column. (This means that if $x_{i_1},\cd,x_{i_c}$ spans the orthogonal complement to $X_j$, then we set $\lam_{i_1}=\cd=\lam_{i_c}.$) After all the operations, we are left with a matrix consisting of nonequal columns that have only one $0$ entry and all the rest entries are $1.$ Moreover, there are no rows with only $1$. Otherwise, revert the operations, we would get an $X_j,$ whose orthogonal complement intersects the orthogonal complement of other $X_i,$ a contradiction. This implies that each row consists of only one entry of $0$ and all other entries are one. Thus, up to rearranging columns and rows, the matrix can be written as $vv^T-\id,$ with $v=(1,1,\cd,1).$ By direct calculation or the matrix determinant lemma, we deduce that $\det(vv^T-\id)=k-1.$ Thus, the reduced system is a determined linear system and has a unique solution. Then we just revert back the operations to get our $\{\lam_l\}.$ The process gives $c\log\lam_l$ either $0$ or depending linearly on $\cms_N\pi_j\du\om_j.$ Thus, we can ensure that $\lam_j$ is smoothly defined on $\TO^{d+c}\times U\s \TO^{d+c}\times N.$
	
	To sum it, we have constructed a smooth metric $g$ and a closed form $\phi,$ both defined on $\TO^{d+c}\times U,$ that calibrate $T=\sum_j \TO_j\times N_j.$ 
	\subsection{Gluing constructions to finish the proof}
	Just apply Lemma 2.7 in \cite{ZL1} and we are done.
	\section{Proof of Theorem \ref{slag}}
	In this section, we will use the setup of $\R^6:(x_1,y_1,\cd,x_3,y_3)$ equipped with the standard metric and standard complex structure of $J\pd_{x_j}=\pd_{y_j},J\pd_{y_j}=-\pd_{x_j}.$
	\subsection{Transition of special Lagrangian form}
	First, let us do a transition construction between the special Lagrangian form and the form we use.
	\begin{lem}\label{ins}
		There exists a smooth metric $\ov{g}$ and a smooth calibration form $\ov{\psi}$ near the $x_3$-axis, so that up to changing the $y_3$ coordinate by constant factors in each interval, $\ov{g}$ coincide with the standard metric on $x_3\in(-\infty,t_0]\cup[t_1,t_2]\cup [t_3,\infty]$ for $t_1<t_2<t_3<t_4$, and the form $\ov{\psi}$ coincides with the standard special Lagrangian form for $x_3\in(-\infty,t_1]\cup (t_4,\infty]$, and equals $dx_1\w dx_2\w dx_3-dy_1\w dy_2\w dx_3$ on $[t_1,t_2].$
	\end{lem}
	We will give the proof in the following two subsubsections. One involves modifying the form to delete or create small $dy_3$ terms.  Another involves rescaling the metric in the $y_3$ direction to make $dy_3$ terms small. 
	\subsubsection{Step 1:creating and deleting $dy_3$ terms}
	We will show how to do the construction on $x_3\in(t_2,t_3)$ and the case for $x_3\in (t_0,t_1)$ is similar. Let $\si$ be a smooth function that is zero on $(\infty,t_2],$ and smoothly monotonically transits to a positive constant $s_0>0$ on $(t_2,t_2+\frac{1}{2}(t_3-t_2)),$ i.e., $\si(x_3)\equiv\ai$ for $x_3>t_2+\frac{1}{2}(t_3-t_2).$ Define 
	\begin{align*}
		\ov{\psi}=&d\bigg(x_3(dx_1\w dx_2-dy_1\w dy_2)-y_3\si(dx_1\w dy_2+dy_1\w dx_2)\bigg)\\
		=&(dx_1\w dx_2-dy_1\w dy_2)\w dx_3-\si(dx_1\w dy_2+dy_1\w dx_2)\w dy_3\\
		&-y_3\si'(dx_1\w dy_2+dy_1\w dx_2)\w dx_3.
	\end{align*}
	By construction $\ov{\psi}$ is closed. Again, $\ov{\psi}$ is of a very special form, i.e., each term only contains either $dx_j$ or $dy_j.$ By 4.2 The torus Lemma in \cite{DHM}, the comass of $\ov{\psi}$ is equal to 
	\begin{align*}
		\max_{\tau_1,\tau_2,\tau_3\in[-\pi,\pi)}\ov{\psi}\bigg((\cos\tau_1\pd_{x_1}+\sin\tau_1\pd_{y_1})\w\cd\w(\cos\tau_3\pd_{x_3}+\sin\tau_3\pd_{y_3})\bigg).
	\end{align*}
	We have
	\begin{align*}
		&\ov{\psi}\bigg((\cos\tau_1\pd_{x_1}+\sin\tau_1\pd_{y_1})\w\cd\w(\cos\tau_3\pd_{x_3}+\sin\tau_3\pd_{y_3})\bigg)\\
		=&\cos(\tau_1+\tau_2)\cos\tau_3-\si\sin(\tau_1+\tau_2)\sin\tau_3-y_3\si'\sin(\tau_1+\tau_2)\cos\tau_3\\
		=&\cos(\tau_1+\tau_2)\cos\tau_3-\sin(\tau_1+\tau_2)(\si\sin\tau_3+y_3\si'\cos\tau_3)\\
		\le&\sqrt{(\cos\tau_3)^2+(\si\sin\tau_3+y_3\si'\cos\tau_3)^2}\\
		=&\sqrt{(\cos\tau_3)^2+\si^2(\sin\tau_3)^2+2y_3\si\si'\sin\tau_3\cos\tau_3+(y_3\si')^2(\cos\tau_3)^2}\\
		=&\sqrt{(1+(y_3\si')^2-\si^2)(\cos\tau_3)^2+y_3\si\si'(2\sin\tau_3\cos\tau_3)+\si^2}\\
		=&\sqrt{(1+(y_3\si')^2-\si^2)\frac{\cos (2\tau_3)+1}{2}+y_3\si\si'\sin(2\tau_3)+\si^2}\\
		=&\sqrt{\frac{1+(y_3\si')^2-\si^2}{2}\cos(2\tau_3)+y_3\si\si'\sin(2\tau_3)+\frac{1+(y_3\si')^2+\si^2}{2}}\\
		\le&\sqrt{\sqrt{(\frac{1+(y_3\si')^2-\si^2}{2})^2+(y_3\si\si')^2}+\frac{1+(y_3\si')^2+\si^2}{2}},
	\end{align*}
	where we have repetitively used the simple fact that $a\cos\ta+b\sin\ta\le \sqrt{a^2+b^2}.$ Now, by multiplying $\si$ with small constants, we can assume without loss of generality that $\no{\si}_{C^1},\no{\si}_{C^0}\le \frac{1}{2}.$ Then for $|y_3|\le 1,$ we have $\frac{3}{4}\le1+(y_3\si')^2-\si^2\le \frac{5}{4},$ and $|y_3\si\si'|\le \frac{1}{4},$ which gives $-\frac{2}{3}\le \frac{y_3\si\si'}{(1+(y_3\si')^2-\si^2)/2}\le \frac{2}{3}.$ Note that $\sqrt{1+h}\le 1+\frac{1}{2}h$ for $0<h<1$, since its Taylor series at $0$ is an alternating series. Thus, we have,
	\begin{align}\label{com}
		\textnormal{comass }\ov{\psi}\le& \sqrt{\frac{1+(y_3\si')^2-\si^2}{2}\bigg(1+\frac{1}{2}\frac{(y_3\si\si')^2}{((1+(y_3\si')^2-\si^2)/2)^2}\bigg)+\frac{1+(y_3\si')^2+\si^2}{2}}\\
		=&\sqrt{1+(y_3\si')^2\bigg(1+\frac{\si^2}{1+(y_3\si')^2-\si^2}\bigg)}.
	\end{align} Let
	\begin{align*}
		\lam(x_3,y_3)=\sqrt{1+(y_3\si')^2\bigg(1+\frac{\si^2}{1+(y_3\si')^2-\si^2}\bigg)}.
	\end{align*} 
	Now, we conformally change the standard metric $\de$ by a conformal factor $\lam^{\frac{1}{3}},$ i.e., defining $$\ov{\de}=\lam^{\frac{2}{3}}\de.$$
	Then $\ov{\psi}$ has comass less than or equal to $1$ in $\ov{\de}.$ 
	Note that when $y_3=0,$ and $\ov{\psi}$ has comass precisely $1,$ calibrating both $\cu{x_1x_2x_3}$ and $\cu{y_1y_2x_3}.$ Also for $x_3>t_2+\frac{1}{2}(t_3-t_2)$, we have $\lam\equiv 1$ and $\ov{\psi}=(dx_1\w dx_2-dy_1\w dy_2)\w dx_3-s_0(dx_1\w dy_2+dy_1\w dx_2)\w dy_3,$ with some $0<s_0<\frac{1}{2}$ by construction of $\si.$ 
	\subsubsection{Step two:transition to $dy_3$ terms small}
	Now let $\ov{\si}$ be a smooth function that is $1$ on $(-\infty,t_2+\frac{1}{2}(t_3-t_2)),$ and smoothly monotonically transits to $s_0$ on $(t_2+\frac{1}{2}(t_3-t_2),t_3),$ i.e., $\ov{\si}\equiv s_0$ for $x_3\ge t_3.$ Rescale the metric $\ov{\de}$ to a metric $\ov{g}$ by multiplying the $dy_3^2$ factor with $\ov{\si}.$ Then in the orthonormal basis in the new metric for $x_3\ge t_2+\frac{1}{2}(t_3-t_2),$, $\ov{\psi}$ becomes $\ov{\psi}=(dx_1\w dx_2-dy_1\w dy_2)\w dx_3-\frac{s_0}{\ov{\si}}(dx_1\w dy_2+dy_1\w dx_2)\w E_3\du,$ where $dx_1,dy_1,\cd,dy_2,dx_3,E_3\du=\ov{\si}dy_3$ is the dual orthonormal basis $1$-forms. Direct calculations using 4.2 The torus lemma in \cite{DHM} shows that for $x_3\ge t_2+\frac{1}{2}(t_3-t_2),$ we have \begin{align*}
		&\textnormal{comass}_{\ov{g}}\ov{\psi}\\\le& \max_{\tau_1,\tau_2,\tau_3\in[-\pi,\pi)}\cos(\tau_1+\tau_2)\cos\tau_3-\frac{s_0}{\ov{\si}}\sin(\tau_1+\tau_2)\sin\tau_3\\\le&|\cos(\tau_1+\tau_2)\cos\tau_3|+|\sin(\tau_1+\tau_2)\sin\tau_3|\\\le &1.
	\end{align*}
	This shows that $\ov{\psi}$ is a calibration in $\ov{g}$. 
	\subsection{Finishing the proof}
	Now we claim that Lemma \ref{ins} combining with \cite{ZL0} and \cite{ZL} gives what we want. First, in $x_3\in(t_1,t_2),$ we can use the similar constructions as in Section 3 of \cite{ZL} to perturb the $x_1x_2x_3$-plane and $y_1y_2x_3$-plane reversely oriented, to get arbitrary singular sets that becomes the $x_3$-axis when getting out of the interval.  Then we can use Lemma \ref{ins} to glue this into the edges on $x_3\in(\infty,t_0)\cup(t_3,\infty)$ in Section 3 of \cite{ZL0}, since here the calibration forms become the special Lagrangian up to constant rescalings of $y_3$ direction. Here we invoke Proposition 6.1 of \cite{ZL0}. In other words, we can glue the conical surface and $y_1y_2x_3$-plane to the union of the tangent plane along the ray and $\cu{y_1y_2x_3}.$ Then we glue the unions of planes to $\cu{x_1x_2x_3}-\cu{y_1y_2x_3}$.
	
	Then if we go through Sections 4 and 5 of \cite{ZL0}, the reasoning carries through and we are done. 
	
\end{document}